\RequirePackage{fix-cm}
\documentclass[smallextended]{svjour3}       
\smartqed  
\usepackage{graphicx}
\usepackage{amsfonts}
\usepackage{mathrsfs}
\usepackage{amsmath}
\begin{document}

\title{Morphisms of tautological control systems
}


\author{Qianqian Xia          
}


\institute{Qianqian Xia \at
              School of Information and Control, Nanjing University of Information Science \& Technology, Nanjing, 210044, China\\
              \email{inertialtec@sina.com}
}

\date{Received: date / Accepted: date}

\maketitle

\begin{abstract}
In this paper, we investigate morphisms of tautological control systems. Given a tautological control system
$\mathfrak{H}$ on the manifold N and a mapping $\Phi: M \to N$, we study existence of tautological control system $\mathfrak{G}$ on the manifold $M$ such that there exists a trajectory-preserving morphism $(\Phi, \Phi^\#)$ from $\mathfrak{G}$ to $\mathfrak{H}$. Sufficient conditions are given such that reachability of $\mathfrak{H}$ implies the reachability of $\mathfrak{G}$. Correspondence between the notion of lifting ordinary control systems and morphisms of tautological control systems are examined. We give an application of the above results to the class of second-order type control systems, where the special structure of second-order type leads to additional results.
\end{abstract}

\section{Introduction}
\label{intro}
Morphisms of ordinary control systems have been studied by several researchers. From a theoretical point of view,
the study of morphisms is interesting since it reveals system structures that must be understood. And it can help us
solve the analysis and design problems for ordinary control systems which are often difficult due to the complicated
nature of the equations describing the systems. Among the existing works the notion of quotient control systems has
been introduced by different authors. In \cite{Ref1}, the analysis of the Lie algebra of an ordinary control systems
leads to a decomposition into smaller systems. In \cite{Ref2}, the problem of reduction for ordinary control systems with symmetries
is studied. In \cite{Ref3}, \cite{Ref4}, Pappas et. al study abstractions of control systems where a constructive procedure
was proposed to compute smaller systems. In \cite{Ref5}, Tabuada and Pappas study quotients of ordinary control systems
by the use of category theory. In \cite{Ref6}, quotients of Lagrangian mechanical control system are investigated.
In a certain sense all the above work focus on the notion of projection. A different research direction is taken
in \cite{Ref7}, where the notion of lifting ordinary control system is introduced. In this paper, we will study
morphisms of tautological control systems introduced by Lewis in \cite{Ref8}. It is to be seen that our research
is in the same direction as in \cite{Ref7}.

Tautological control systems proposed by Lewis provide us with a framework that models systems in a manner that
does not make any choice of parameterisation by control. This framework is able to handle a variety of degrees of
regularity: finitely differentiable, Lipschitz, smooth and especially real analytic.  Due to the use of sheaf theory,
the framework can deal with distinctions between local and global, which are unnoticed in current control theory.
The problem to be explored in this paper goes as follows.

Given a $C^\nu$ tautological control system $\mathfrak{H}=(N, \mathscr{G})$
on the $C^r$ manifold $N$ and a mapping $\Phi \in C^{r}(M, N)$ where $M$ is a $C^r$ manifold. We will be interested in
answering the main questions:

(i) Does there exist a tautological control system $\mathfrak{G}=(M, \mathscr{F})$ such that there exists a trajectory-preserving
morphism $(\Phi, \Phi^\#)$ from $\mathfrak{G}$ to $\mathfrak{H}$? (ii) When does the reachability of
$\mathfrak{H}$ imply the reachability of $\mathfrak{G}$?

See Section 2 for the definitions of tautological control system and trajectory-preserving
morphism, and the precise information about $\nu$ and $r$ above.

Similar question has been studied in \cite{Ref7} in the context of ordinary control systems, where the systems are
assumed to be only $C^1$ control systems. Here we
deal with tautological control systems. It is know in \cite{Ref8} that given an ordinary control system, we can associate
 a tautological control system to it. With the correspondence between trajectories of an ordinary control system and
 its associated tautological control system as claimed in \cite{Ref8}, then the useful system theoretic properties can
 be found in the framework of tautological control system without having to go back to the ordinary control system
 framework.  See Section 2 for detailed interpretations of the correspondence between ordinary control system and tautological
 control system.

Our work here is different from the work in \cite{Ref7}. First, the answer to the question in \cite{Ref7} is localised, and distinctions between local and
 global are not well dealt with. Our work tries to understand the passage from local to global and patch together local
 constructions to give global constructions. It is to be seen that with the aid of sheaf language, globally defined
 control constructions can be got. Second, our work also investigate the case of real analytic systems.
 Real analyticity plays an important role in geometric control theory (e.g. Orbit Theorem \cite{Ref9}).
 However, real analyticity hasn't been well understood in control theory. One reason is that the underlying geometry is a lot different from smooth differential geometry,
 and the techniques of real analytic differential geometry are difficult to learn and to learn to apply.
 In this paper we try to make some discussions on real analytic system with the aid of sheaf theory.

Our work is in the framework  of tautological control system introduced by Lewis and Jafarpour
in \cite{Ref8} and \cite{Ref11}. Here we focus on the structure of trajectory-preserving morphism of
tautological control systems. The main result (see Theorem 3) provides a sufficient condition for the reachability
of  $\mathfrak{H}$ to imply the reachability of $\mathfrak{G}$. Compared with the framework of ordinary control
system, it is to be seen that there is something different in the framework of tautological control system, where
there is no choice of parameterisation by control.

The paper is organized as follows. In Section 2 we provide some basic notions and present some known results
that will be needed for the subsequent work. Section 3 contains the main results.  Question (i) is investigated
for systems with different degrees of regularity.  For question (ii) we give sufficient conditions for
the reachability of $\mathfrak{H}$ to imply the reachability of $\mathfrak{G}$. An extended result for small-time local
controllability is presented. Correspondence between
the notion of lifting ordinary control systems and morphisms of tautological control systems are examined. In Section 4 an example of second-order type control systems is studied. Compared with the main theorem for general tautological
control systems in Section 3,  we get supplementary results for this class of systems due to the special structure properties of second-order type.
\section{Preliminaries}
\label{sec:1}
Throughout the paper, manifolds are supposed to be second-countable Hausdorff manifolds. We will use the letter $n$ to denote the dimension of the manifold $M$. The tangent bundle of a manifold $M$
is denoted by $\pi_{\mathrm{TM}}: TM \rightarrow M$ and the cotangent bundle by $\pi_{\mathrm{T^{*}M}}: T^{*}M
\rightarrow M$. The derivative of a differentiable map $\Phi: M \to N$ is denoted by $T\Phi: TM \rightarrow TN$, with $T_{x}\Phi=
T\Phi|T_{x}M$. If $I \in \mathbb{R}$ is an interval and if $\xi: I \rightarrow M$ is a curve that is differentiable at $t \in I$,
we denote the tangent vector field to the curve at $t$ by $\xi{'}(t)=T_{t}\xi(1)$. Let $r \in \mathbb{Z}_{\geq 0} \cup \{\infty, \omega\}$.
The set of sections of a vector bundle $E \rightarrow M$ of class $C^r$ is denoted by $\mathrm{\Gamma}^{r}(E)$.  We denote by $\mathrm{C}^{r}(M, N)$ the set of mappings of class $C^{r}$ from $M$ to $N$.

For an interval $I$ and a topological space $\mathcal{X}$, a curve $\gamma: I \rightarrow \mathcal{X}$ is measurable
if $\gamma^{-1}(\mathcal{O})$ is Lebesgue measurable for every open $\mathcal{O} \subseteq \mathcal{X}$. By $\mathrm{L}^{\infty}(I; \mathcal{X})$
we denote the measurable curves $\gamma: I \rightarrow \mathcal{X}$ for which there exists a compact set $K \subseteq \mathcal{X}$
with $\lambda(\{t \in I| \gamma(t) \notin K\})=0$, where $\lambda$ is a Lebesgue measure, i.e., $\mathrm{L}^{\infty}(I; \mathcal{X})$ is the set of essentially
bounded curves. By $\mathrm{L}^{\infty}_\mathrm{loc}(I; \mathcal{X})$ we denote the locally essentially bounded curves,
meaning those measurable curves whose restrictions to compact subintervals are essentially bounded.

Let $m \in \mathbb{Z}_{\geq 0}$ and $m{'} \in \{0, \mathrm{lip}\}$, let $\nu \in \{m+m{'}, \infty, \omega\}$, and let
$r \in \{\infty, \omega\}$, depending from the context. For a manifold $M$ of class $C^r$ and an interval $I \subseteq \mathbb{R}$,
we denote the set of Caratheodory time-varying vector field of class $C^\nu$ by
$\mathrm{CF\Gamma}^{\nu}(I; TM)$, the set of locally integrally $C^\nu$-bounded time-varying vector fields by
$\mathrm{LI\Gamma}^{\nu}(I; TM)$, and the set of locally essentially $C^\nu$-bounded time-varying vector fields by
$\mathrm{LB\Gamma}^{\nu}(I; TM)$. For a manifold $M$ of class $C^r$ and a topological space $\mathcal{P}$, we denote the set of separately parameterised vector fields of class $C^\nu$ by
$\mathrm{SP\Gamma}^{\nu}(\mathcal{P}; TM)$, the set of jointly  parameterised vector fields of class $C^\nu$ by
$\mathrm{JP\Gamma}^{\nu}(\mathcal{P}; TM)$. For detailed explanations for these concepts, we refer to \cite{Ref8}.

\begin{definition}\cite{Ref8}
Let $m \in \mathbb{Z}_{\geq 0}$ and $m{'} \in \{0, \mathrm{lip}\}$, let $\nu \in \{m+m{'}, \infty, \omega\}$,
and let $r \in \{\infty, \omega\}$, as required. Let $M$ be a manifold of class $C^{r}$.  A presheaf of sets of $C^{\nu}$-vector fields is an assignment to each open set $\mathcal{U} \subseteq M$ a subset $\mathscr{F}(\mathcal{U})$ of $\Gamma^{\nu}(T\mathcal{U})$ with the property that, for open sets $\mathcal{U}, \mathcal{V} \subseteq M$ with $\mathcal{V} \subseteq \mathcal{U}$, the map
\begin{eqnarray}
r_{\mathcal{U}, \mathcal{V}}&:& \mathscr{F}(\mathcal{U}) \rightarrow \Gamma^{\nu}(T\mathcal{V})\nonumber\\
&&X \rightarrow X|\mathcal{V}\nonumber
\end{eqnarray}
takes values in $\mathscr{F}(\mathcal{V})$. Elements of $\mathscr{F}(\mathcal{U})$ are called local sections of
$\mathscr{F}$ over $\mathcal{U}$.
\end{definition}
\begin{example}
Let $m \in \mathbb{Z}_{\geq 0}$ and $m{'} \in \{0, \mathrm{lip}\}$, let $\nu \in \{m+m{'}, \infty, \omega\}$, and let $r \in \{\infty, \omega\}$, as required. Let $M$ be a manifold of class $C^{r}$. If $\mathscr{X} \subseteq \Gamma^{\nu}(TM)$ is any family of vector fields on $M$, then we can define an associated presheaf $\mathscr{F}_{\mathscr{X}}$ of sets of vector fields by
\[\mathscr{F}_{\mathscr{X}}(\mathcal{U})=\{X|\mathcal{U} \ | \ X \in \mathscr{X}\}.\]
A presheaf of this sort will be called globally generated.
\end{example}

With the above definition of presheaf in mind, we have the following definition of tautological control system.
\begin{definition}\cite{Ref8}
Let $m \in \mathbb{Z}_{\geq 0}$ and $m^{'} \in \{0, \mathrm{lip}\}$, let $\nu \in \{m+m^{'}, \infty, \omega\}$, and let
$r \in \{\infty, \omega\}$, as required. A $C^{\nu}$-tautological control system is a pair
$\mathfrak{G}=(M, \mathscr{F})$ where $M$ is a manifold of class $C^{r}$ whose elements are called states and where
$\mathscr{F}$ is a presheaf of sets of $C^{\nu}$ vector fields on $M$. A tautological control system $\mathfrak{G}=(M, \mathscr{F})$
is globally generated if $\mathscr{F}$ is globally generated.
\end{definition}

Now we introduce a notion of a trajectory for a tautological control system.  Since trajectories are associated to "open-loop systems",
we first discuss "open-loop systems". Let $\mathfrak{G}=(M, \mathscr{F})$ be a $C^{\nu}$-tautological control system. An open-loop system for $\mathfrak{G}$
is a triple $\mathfrak{G}_{ol}=(X, \mathbb{T}, \mathcal{U})$ where $\mathbb{T} \subseteq \mathbb{R}$ is an interval called
the time-domain, $\mathcal{U} \subseteq M$ is open and $X \in \mathrm{LI\Gamma^{\nu}}(\mathbb{T}; \mathscr{F}(\mathcal{U}))$.
An open-loop subfamily for $\mathfrak{G}$ is an assignment, to each interval $\mathbb{T} \subseteq \mathbb{R}$ and
each open set $\mathcal{U} \subseteq M$, a subset
$\mathscr{O}_{\mathfrak{G}}(\mathbb{T}, \mathcal{U}) \subseteq \mathrm{LI\Gamma^{\nu}}(\mathbb{T}; \mathscr{F}(\mathcal{U}))$ with the property that,
if $(\mathbb{T}_1, \mathcal{U}_1)$ and $(\mathbb{T}_2, \mathcal{U}_2)$ are such that $\mathbb{T}_1 \subseteq \mathbb{T}_2,\mathcal{U}_1 \subseteq \mathcal{U}_2$,
then \[\{X|\mathbb{T}_1 \times \mathcal{U}_1 \ | \ X \in \mathscr{O}_{\mathfrak{G}}(\mathbb{T}_2, \mathcal{U}_2)\} \subseteq \mathscr{O}_{\mathfrak{G}}(\mathbb{T}_1, \mathcal{U}_1).\]
The full subfamily for $\mathfrak{G}$ is the open-loop subfamily $\mathscr{O}_{\mathfrak{G}, \mathrm{full}}$ defined by
\[\mathscr{O}_{\mathfrak{G}, \mathrm{full}}(\mathbb{T}, \mathcal{U})= \mathrm{LI\Gamma^{\nu}}(\mathbb{T}; \mathscr{F}(\mathcal{U})).\]
The locally essentially bounded subfamily for $\mathfrak{G}$ is the open-loop subfamily $\mathscr{O}_{\mathfrak{G}, \infty}$
defined by
 \[\mathscr{O}_{\mathfrak{G}, \infty}(\mathbb{T}, \mathcal{U})=\{X \in \mathscr{O}_{\mathfrak{G}, \mathrm{full}}(\mathbb{T}, \mathcal{U}) \ | \ X \in
\mathrm{LB\Gamma}^{\nu}(\mathbb{T}; T\mathcal{U})\}.\]
The locally essentially compact subfamily for $\mathfrak{G}$ is the open-loop subfamily $\mathscr{O}_{\mathfrak{G}, \mathrm{cpt}}$ defined by
\begin{eqnarray}
\mathscr{O}_{\mathfrak{G}, \mathrm{cpt}}(
\mathbb{T}, \mathcal{U})&=&\{X \in \mathscr{O}_{\mathfrak{G}, \mathrm{full}}(\mathbb{T}, \mathcal{U}) \ |
\ \text{for every compact subinterval
}\ \mathbb{T{'} \subseteq \mathbb{T}}\nonumber\\
&&\text{there exists a compact} \ K \subseteq \mathrm{\Gamma}^{\nu}(T\mathcal{U}) \nonumber\\
&&\text{such that} \ X(t) \in K \
\text{for almost every} \ t \in \mathbb{T{'}}\}.\nonumber
\end{eqnarray}
The piecewise constant subfamily for $\mathfrak{G}$ is the open-loop subfamily $\mathscr{O}_{\mathfrak{G}, \mathrm{pwc}}$ defined by
\[\mathscr{O}_{\mathfrak{G}, \mathrm{pwc}}(\mathbb{T}, \mathcal{U})=\{X \in \mathscr{O}_{\mathfrak{G}, \mathrm{full}}(\mathbb{T}, \mathcal{U})
\ | \
t \rightarrow X(t) \ \text{is piecewise constant}\}.\]
A $(\mathbb{T}, \mathcal{U})$-trajectory for $\mathscr{O}_{\mathfrak{G}}$ is a curve $\xi: \mathbb{T}\rightarrow \mathcal{U}$
such that $\xi{'}(t)=X(t, \xi(t))$ for some $X \in \mathscr{O}_{\mathfrak{G}}(\mathbb{T}, \mathcal{U})$. A plain trajectory for $\mathscr{O}_{\mathfrak{G}}$
is a curve that is a $(\mathbb{T}, \mathcal{U})$-trajectory for $\mathscr{O}_{\mathfrak{G}}$ for some time domain
$\mathbb{T}$ and some open set $ \mathcal{U} \subseteq M$. We denote the set of $(\mathbb{T}, \mathcal{U})$-trajectories
for $\mathscr{O}_{\mathfrak{G}}$ by $\mathrm{Traj}(\mathbb{T}, \mathcal{U}, \mathscr{O}_{\mathfrak{G}})$ and
the set of trajectories for $\mathscr{O}_{\mathfrak{G}}$ by $\mathrm{Traj}(\mathscr{O}_{\mathfrak{G}})$.
Then according to \cite{Ref11}, we can assert that for every $x \in M, s \in \mathbb{R}$ and
$X \in \mathscr{O}_{\mathfrak{G}, \mathrm{full}}(\mathbb{T}, \mathcal{U})$, there exists an open interval
$J_{X}(s,x) \subseteq \mathbb{T}$ containing $s$ such that  $\xi{'}(t)=X(t, \xi(t)), \xi(s)=x$ for almost all
$t \in J_{X}(s,x)$.

In the following, we introduce the concept of morphism of tautological control systems.
Let $m \in \mathbb{Z}_{\geq 0}$ and $m{'} \in \{0, \mathrm{lip}\}$, let $\nu \in \{m+m{'}, \infty, \omega\}$,
and let $r \in \{\infty, \omega\}$, as required. Let $\mathfrak{G}=(M, \mathscr{F})$ be a $C^{\nu}$-tautological control
system, let $N$ be a $C^r$-manifold, and let $\Phi \in C^{r}(M, N).$ The direct image of $\mathfrak{G}$ by $\Phi$ is the
tautological control system $\Phi_{*}\mathfrak{G}=(N, \Phi_{*}\mathscr{F})$ defined by
$\Phi_{*}\mathscr{F}(\mathcal{V})=\mathscr{F}(\Phi^{-1}(\mathcal{V}))$ for $\mathcal{V} \subseteq N$ open.
\begin{definition}\cite{Ref8}\label{Def3}
Let $m \in \mathbb{Z}_{\geq 0}$ and $m{'} \in \{0, \mathrm{lip}\}$, let $\nu \in \{m+m{'}, \infty, \omega\}$,
and let $r \in \{\infty, \omega\}$, as required. Let $\mathfrak{G}=(M, \mathscr{F})$ and $\mathfrak{H}=(N, \mathscr{G})$
be $C^{\nu}$-tautological control systems. A morphism from $\mathfrak{G}$ to $\mathfrak{H}$ is a pair $(\Phi, \Phi^{\#})$ such that \\
(i) $\Phi \in C^{r}(M, N)$ and \\
(ii) $\Phi^{\#}=(\Phi^{\#}_{\mathcal{V}})_{{\mathcal{V} \mathrm{open}}}$ is a family of mappings $\Phi^{\#}_{\mathcal{V}}: \mathscr{G}(\mathcal{V}) \rightarrow \Phi_{*}\mathscr{F}(\mathcal{V}), \mathcal{V} \subseteq N$ defined as follows:

(a) there exists a family $L_{\mathcal{V}} \in \mathrm{L}(\mathrm{\Gamma^{\mathcal{\nu}}}(T\mathcal{V}); \mathrm{\Gamma^{\mathcal{\nu}}}(T(\Phi^{-1}(\mathcal{V}))) )$ of continuous linear mappings satisfying $L_{{\mathcal{V}{'}}}=L_{\mathcal{V}}|\mathrm{\Gamma^{\mathcal{\nu}}}(T\mathcal{V}{'})$ if $\mathcal{V}, \mathcal{V}{'} \subseteq N$ are open with $\mathcal{V}{'} \subseteq \mathcal{V}$;

(b) $\Phi^{\#}_{\mathcal{V}}=L_{\mathcal{V}}|\mathscr{G}(\mathcal{V})$.
\end{definition}

From the point of view of control theory, one wishes to restrict the above definition further to account for the
fact that morphisms ought to preserve trajectories.
\begin{definition}\cite{Ref8}
Let $m \in \mathbb{Z}_{\geq 0}$ and $m{'} \in \{0, \mathrm{lip}\}$, let $\nu \in \{m+m{'}, \infty, \omega\}$,
and let $r \in \{\infty, \omega\}$, as required. Let $\mathfrak{G}=(M, \mathscr{F})$ and $\mathfrak{H}=(N, \mathscr{G})$
be $C^{\nu}$-tautological control systems. A morphism $(\Phi, \Phi^{\#})$  from $\mathfrak{G}$ to $\mathfrak{H}$ is trajectory-preserving if, for each time-domain
$\mathbb{T}$, each open $\mathcal{V} \subseteq N$, and each $Y \in \mathrm{LI\Gamma}^{\nu}(\mathbb{T};
\mathscr{G}(\mathcal{V}))$, any integral curve $\xi: \mathbb{T}{'} \rightarrow \Phi^{-1}(\mathcal{V})$
for the time-varying vector field $t \rightarrow \Phi^{\#}(Y_t)$ defined on $\mathbb{T}{'} \subseteq \mathbb{T}$
has the property that $\Phi \circ \xi$ is an integral curve for Y.
\end{definition}

We have the following characterisation of trajectory-preserving morphisms.
\begin{proposition}\cite{Ref8}
Let $m \in \mathbb{Z}_{\geq 0}$ and $m{'} \in \{0, \mathrm{lip}\}$, let $\nu \in \{m+m{'}, \infty, \omega\}$,
and let $r \in \{\infty, \omega\}$, as required. Let $\mathfrak{G}=(M, \mathscr{F})$ and $\mathfrak{H}=(N, \mathscr{G})$
be $C^{\nu}$-tautological control systems. A morphism $(\Phi, \Phi^{\#})$  from $\mathfrak{G}$ to $\mathfrak{H}$ is trajectory-preserving if and only if, for each open
$\mathcal{V} \subseteq N$, each $Y \in \mathscr{G}(\mathcal{V})$, each $y \in \mathcal{V}$, and each $x \in \Phi^{-1}(y)$,
we have $T_{x}\Phi(\Phi^{\#}(Y)(x))=Y(y)$.
\end{proposition}

Based on the above proposition, we then introduce the following definition, which will be essential in our main reachability
result to be proved in Section 3.
\begin{definition}
Let $m \in \mathbb{Z}_{\geq 0}$ and $m{'} \in \{0, \mathrm{lip}\}$, let $\nu \in \{m+m{'}, \infty, \omega\}$,
and let $r \in \{\infty, \omega\}$, as required. Let $\mathfrak{G}=(M, \mathscr{F})$ and $\mathfrak{H}=(N, \mathscr{G})$
be $C^{\nu}$-tautological control systems. We say that the morphism $(\Phi, \Phi^{\#})$  from $\mathfrak{G}$ to $\mathfrak{H}$ is
global in time if for each $s \in \mathbb{R} $, each open
$\mathcal{V} \subseteq N$, each $Y \in \mathscr{G}(\mathcal{V})$, and each $x \in \Phi^{-1}(\mathcal{V})$,
we have $J_{\Phi^{\#}(Y)}(s, x)=J_{Y}(s, \Phi(x))$.
\end{definition}

In the remaining part of this section, we will discuss correspondences between ordinary control systems and tautological
control systems.
\begin{definition}\cite{Ref8}
Let $m \in \mathbb{Z}_{\geq 0}$ and $m{'} \in \{0, \mathrm{lip}\}$, let $\nu \in \{m+m{'}, \infty, \omega\}$, and let
$r \in \{\infty, \omega\}$, as required. A $C^{\nu}$-control system is a triple $\Sigma=(M, F, \mathfrak{C})$, where

(i) M is a $C^r$-manifold whose elements are called states.

(ii) $\mathfrak{C}$ is a topological space called the control set and

(iii) $F \in \mathrm{JP\Gamma}^{\nu}(\mathfrak{C};  TM)$.
\end{definition}
\begin{proposition}\cite{Ref8}\label{prop2}
Let $m \in \mathbb{Z}_{\geq 0}$ and $m{'} \in \{0, \mathrm{lip}\}$, let $\nu \in \{m+m{'}, \infty, \omega\}$, and let
$r \in \{\infty, \omega\}$, as required. Let  $\Sigma=(M, F, \mathfrak{C})$ be a $C^{\nu}$-control system. If $\mu \in
\mathrm{L}^{\infty}_\mathrm{loc}(\mathbb{T}; \mathfrak{C})$ then $F^{\mu} \in \mathrm{LB\Gamma}^{\nu}(\mathbb{T}; TM)$, where
$F^{\mu}: \mathbb{T} \times M \rightarrow TM$ is defined by $F^{\mu}(t,x)=F(x, \mu(t)).$
\end{proposition}
Let  $\Sigma=(M, F, \mathfrak{C})$ be a $C^{\nu}$-control system. For an interval $\mathbb{T} \subseteq \mathbb{R}$,
a $\mathbb{T}$-trajectory is a locally absolutely continuous curve $\xi: \mathbb{T} \rightarrow M$ for which there exists
$\mu \in \mathrm{L}^{\infty}_\mathrm{loc}(\mathbb{T}; \mathfrak{C})$ such that
$\xi{'}(t)=F(\xi(t), \mu(t)), a.e. \ t \in \mathbb{T}.$  The set of $\mathbb{T}$-trajectories we denote by
$\mathrm{Traj}(\mathbb{T}, \Sigma)$. If $\mathcal{U}$ is open, we denote by $\mathrm{Traj}(\mathbb{T},\mathcal{U},\Sigma)$
those trajectories taking value in $\mathcal{U}$. To this control system we associate the
$C^{\nu}$ tautological control system $\mathfrak{G}_{\Sigma}=(M, \mathscr{F}_{\Sigma})$ by
\[\mathscr{F}_{\Sigma}(\mathcal{U})=\{F^{u}|\mathcal{U}\in \mathrm{\Gamma}^{\nu}(T\mathcal{U})\ | \ u \in \mathfrak{C}\}.\]
The presheaf of sets of vector fields in this case is of the globally generated variety.

Suppose we have a $C^{\nu}$-tautological control system $\mathfrak{G}=(M, \mathscr{F})$
where $\mathscr{F}$ is globally generated. We define a $C^{\nu}$-control system
$\Sigma_{\mathfrak{G}}=(M, F_\mathscr{F}, \mathfrak{C}_{\mathscr{F}})$ as follows. We take $\mathfrak{C}_{\mathscr{F}}=
\mathscr{F}(M)$, i.e. the control set is our family of globally defined vector fields and the topology is that induced
from $\mathrm{\Gamma}^{\nu}(TM)$. We define
\begin{eqnarray}
F_{\mathscr{F}}&:& M \times \mathfrak{C}_{\mathscr{F}} \rightarrow TM\nonumber\\
&&(x, X) \rightarrow X(x)\nonumber
\end{eqnarray}
\begin{proposition}\cite{Ref8}
Let $m \in \mathbb{Z}_{\geq 0}$ and $m{'} \in \{0, \mathrm{lip}\}$, let $\nu \in \{m+m{'}, \infty, \omega\}$, and let
$r \in \{\infty, \omega\}$, as required. Let $\mathfrak{G}=(M, \mathscr{F})$ be a $C^{\nu}$-tautological control
system. Let  $\Sigma=(M, F, \mathfrak{C})$ be a $C^{\nu}$-control system. Then the following statements hold:\\
(i) if $\mathfrak{G}$ is globally generated, then $\mathfrak{G}_{\Sigma_{\mathfrak{G}}}=\mathfrak{G}$;\\
(ii) if the map $u \rightarrow F^{u}$ from $\mathfrak{C}$ to $\mathrm{\Gamma}^{\nu}(TM)$ is injective and open onto its image,
then $\Sigma_{\mathfrak{G}_{\Sigma}}=\Sigma.$
\end{proposition}
\begin{proposition}\cite{Ref8}
Let $m \in \mathbb{Z}_{\geq 0}$ and $m{'} \in \{0, \mathrm{lip}\}$, let $\nu \in \{m+m{'}, \infty, \omega\}$, and let
$r \in \{\infty, \omega\}$, as required.  Let  $\Sigma=(M, F, \mathfrak{C})$ be a $C^{\nu}$-control system with
$\mathfrak{G}_{\Sigma}$ the associated $C^{\nu}$-tautological control system. Then the following statements hold:\\
(i) $\mathrm{Traj}(\mathbb{T}, \mathcal{U}, \Sigma) \subseteq \mathrm{Traj}
(\mathbb{T}, \mathcal{U}, \mathscr{O}_{\mathfrak{G}_{\Sigma}, \text{cpt}})$;\\
(ii) if the map $u \rightarrow F^{u}$ is injective and proper, then $\mathrm{Traj}
(\mathbb{T}, \mathcal{U}, \mathscr{O}_{\mathfrak{G}_{\Sigma}, \text{cpt}}) \subseteq \mathrm{Traj}(\mathbb{T}, \mathcal{U}, \Sigma)$;\\
(iii) if $\mathfrak{C}$ is a Suslin topology space and if F is proper, then
$\mathrm{Traj}(\mathbb{T}, \mathcal{U}, \mathscr{O}_{\mathfrak{G}_{\Sigma}, \infty}) \subseteq \mathrm{Traj}(\mathbb{T}, \mathcal{U}, \Sigma)$;\\
(iv) if, in addition, $\nu \in \{\infty, \omega\}$, then we may replace $\mathrm{Traj}
(\mathbb{T}, \mathcal{U}, \mathscr{O}_{\mathfrak{G}_{\Sigma}, \text{cpt}})$ with $\mathrm{Traj}
(\mathbb{T}, \mathcal{U}, \mathscr{O}_{\mathfrak{G}_{\Sigma}, \infty})$ in statements (i) and (ii).
\end{proposition}
\section{Morphisms of tautological control systems}
\label{sec:2}
\subsection{Existence of tautological control system}
\label{sec:3}
\begin{theorem}\label{thm1}
Let $m \in \mathbb{Z}_{\geq 0}$ and $m{'} \in \{0, \mathrm{lip}\}$, let $\nu \in \{m+m{'}, \infty\}$,
and let $r=\infty$, as required. Let $\mathfrak{H}=(N, \mathscr{G})$ be a $C^{\nu}$-tautological control system that is
globally generated by a family of linearly independent vector fields, i.e. pointwise-independent everywhere.
Let $M$ be a $C^{r}$ manifold and $\Phi: M \to N$ be a submersion. Then there exists a globally generated $C^{\nu}$-tautological control system
$\mathfrak{G}=(M, \mathscr{F})$ such that there exists  a trajectory-preserving morphism $(\Phi, \Phi^{\#})$ from $\mathfrak{G}$ to $\mathfrak{H}$.
\end{theorem}
\begin{proof}
Due to the local representatives for the submersion $\Phi$, for each $Y \in \mathscr{G}(N)$,
and each $x_{\alpha} \in M$, there exist a neighbourhood $\mathcal{U}_{\alpha} \subseteq M$ of $x_{\alpha}$
and a $C^{\nu}$-vector field $X_{\alpha}$ defined on $\mathcal{U}_{\alpha}$, such that $T_{x}\Phi(X_{\alpha})=Y(\Phi(x))$
for each $x \in \mathcal{U}_{\alpha}$. Then we get an open covering
$\{\mathcal{U}_{\alpha}\}$ of $M$. Since $M$ is a second-countable Hausdorff manifold, let
$\{(\mathcal{W}_i, g_i)\}$ be a partition of unity subordinated to the open covering $\{\mathcal{U}_{\alpha}\}$ such that
each open set $\mathcal{W}_i$ is a subset of an open set $\mathcal{U}_{\alpha(i)}$. Let $X$ be defined by
\[X(x)=\sum_{i}g_{i}X_{\alpha(i)}(x)\]
a finite sum at each $x \in M$, then $X$ is a $C^{\nu}$-vector field on $M$ such that
\[T_{x}\Phi(X)=\sum_{i}g_{i}T_{x}\Phi(X_{\alpha(i)}(x)))=\sum_{i}g_{i}Y(\Phi(x))=Y(\Phi(x)),\]
for every $x \in M$. Let $\Phi^{\#}(Y)=X$. Then for all $Y \in \mathscr{G}(N)$, this gives a family of $C^{\nu}$-vector fields
on $M$ denoted by $\mathscr{X}$. Let $\mathscr{F}$ be globally generated such that $\mathscr{F}(M)=\mathscr{X}$. It is easy to verify
that $(\Phi, \Phi^{\#})$ is a trajectory-preserving morphism from $\mathfrak{G}$ to $\mathfrak{H}$.
\end{proof}
\remark In the above theorem the assumption on $\mathfrak{H}=(N, \mathscr{G})$ that is
globally generated by a family of linearly independent vector fields is to make sure condition (ii) in Definition \ref{Def3} can be satisfied.
In \cite{Ref7}, existence of a lifted system is discussed in the context of ordinary control system, where a transpose mapping
$(d\Phi)^{T}$  of the differential mapping $d\Phi$ is used to construct the vector fields of the control system. We should point out
that the transpose mapping $(d\Phi)^{T}$ defined in a coordinate chart can not be defined globally as claimed in \cite{Ref7}.
When applicable, partition of unity is unavoidable to patch local constructions to give global constructions.

\remark From the point view of control theory, the above assumption on linear independence of vector fields
is restrictive. However, consider the control-affine system $\Sigma=(M, F, \mathfrak{C})$
with $F(x,u)=f_0(x)+\sum_{a=1}^{k}u^{a}f_a(x).$  Let $\bar {\mathfrak{H}}$ be its associated tautological control system.
Suppose that $f_0, f_1, \cdots, f_k$ are linear independent.
Then with similar arguments as given in the above theorem, we know that there exists a globally generated
tautological control system $\bar {\mathfrak{G}}=(M, \mathscr{F})$, which is associated to a control-affine system, such that there exists
a trajectory-preserving morphism $(\Phi, \Phi^{\#})$ from $\bar {\mathfrak{G}}$ to $\bar {\mathfrak{H}}$.

In the following we will discuss the case of real analytic system.  To get global results, we strengthen
assumptions on the structure of the mapping $\Phi$ because of the lack of partition of unity.
\begin{theorem}
Let $\mathfrak{H}=(N, \mathscr{G})$ be a $C^{\omega}$-tautological control system. Let
$\Phi: M \to N$ be a real analytic vector bundle. Then there exists a $C^{\omega}$-tautological control system
$\mathfrak{G}=(M, \mathscr{F})$ such that there exists  a trajectory-preserving morphism $(\Phi, \Phi^{\#})$ from $\mathfrak{G}$ to $\mathfrak{H}$.
\end{theorem}
\begin{proof}
Since $\Phi: M \to N$ be a real analytic vector bundle, then according to \cite{Ref11}, there exists a real analytic
linear connection on $M$. Then for each open $\mathcal{V} \subseteq N$, each $Y \in \mathscr{F}(\mathcal{V})$, there exists
a unique $C^{\omega}$- vector field $X$ on $\Phi^{-1}(\mathcal{V})$ that is a horizontal lift of $Y$.
Let $\Phi^{\#}_{\nu}(Y)=X$. Since the mapping of horizontal lift of vector field is a point to point linear mapping, it is
easy to check that the condition (ii) in Definition \ref{Def3} can be satisfied. That is, $(\Phi, \Phi^{\#})$  is
a trajectory-preserving morphism from $\mathfrak{G}$ to $\mathfrak{H}$.
\end{proof}
\remark Given a $C^{\omega}$-tautological control system $\mathfrak{H}=(N, \mathscr{G})$, let $\Phi: M \to N$ be a submersion,
where $M$ is a real analytic manifold. Generally we can not assert the existence of $C^{\omega}$-tautological control system $\mathfrak{G}=(M, \mathscr{F})$
that satisfies the requirement of the above theorem. One reason is that in the real analytic case,
to get globally defined sections, the useful tool is coherent analytic sheaves instead of partition of unity.
We say that a $C^{\omega}$-vector field $X$ (function $f$) on the open set $U$ of $M$ is projectable
if there exists a $C^{\omega}$-vector field $Y$ (function $g$) on the open set $\Phi(U)$ of $N$ such that
$T\Phi(X)=Y$ ($g \circ \Phi=f$). Let $\mathscr{G}_{\Phi}^{\omega}(M)$ ($\mathscr{C}_{\Phi}^{\omega}(M)$)
denotes the presheaf of projectable vector fields (functions) on $M$ of class $C^{\omega}$.
Then $\mathscr{C}_{\Phi}^{\omega}(M)$ is a sheaf of rings over $M$ and $\mathscr{G}_{\Phi}^{\omega}(M)$ is a sheaf of $\mathscr{C}_{\Phi}^{\omega}(M)$-modules over $M$ that is not locally finitely generated.
In view of this, for any point $x \in M$, consider the local representatives for the submersion on the neighbourhood $U_x$ of $x$, $\Phi(x_1, \cdots, x_n)=(x_1, \cdots, x_m)$.
Let $\mathcal{F}(U_x)$ be the set of projectable vector fields on $U$ of class $C^{\omega}$ such that
$x_j(X)=0$, where $x_j$ denotes the $j$th coordinate of $X \in \mathcal{F}(U), j=m+1, \cdots, n.$
Then the set $B$ consisting of all open sets $V \subseteq U_x$ for all $x \in M$ is a basis for the topology on $M$.
It can be proved that under the restricting map $r_{U,V}$ we get a sheaf $\mathcal{F}$ of sets on $B$,
which gives a unique sheaf $\mathcal{F}^{\text ext}$ of $\mathscr{C}_{\Phi}^{\omega}(M)$-modules over $M$.
If $M$ is locally compact, it can be seen that $\mathcal{F}^{\text ext}$  is locally finitely generated.
However, analogous proof for Oka Coherence Theorem can not be applied to $\mathcal{F}^{\text ext}$
to prove the locally finitely generated properties for the sheaf of relations of sections over open set $U \subseteq M$.
\subsection{Reachability preserving morphisms of tautological control systems}
\label{sec:4}
In Section 2, a notion of a trajectory for a tautological control system has been introduced. It turns out
there is a limitation of  this sort of definition. In order to being able to use the full power of the tautological control system framework,
we will introduce an appropriately  extended notion of trajectory as given in \cite{Ref8}.
We consider the sheaf of time-varying
vector fields on $\mathbb{T} \times M$, where $\mathbb{T}$ is an interval. Let $m \in \mathbb{Z}_{\geq 0}$ and $m{'} \in \{0, \mathrm{lip}\}$, let $\nu \in \{m+m{'}, \infty, \omega\}$,
and let $r \in \{\infty, \omega\}$, let $\mathbb{T}{'} \subseteq \mathbb{T}$
and $\mathcal{U} \subseteq M$ be open. Define the sheaf $\mathrm{LI\mathscr{G}^{\nu}(\mathbb{T}, TM)}$ by
\[\mathrm{LI\mathscr{G}^{\nu}(\mathbb{T}, TM)}(\mathbb{T}{'} \times \mathcal{U})=L^{1}(\mathbb{T}{'},
\mathrm{\Gamma}^{\nu}(T\mathcal{U})),\]
where $L^{1}(\mathbb{T}{'}, \mathrm{\Gamma}^{\nu}(T\mathcal{U}))$ denotes the space of Bochner integrable functions with
values in $\mathrm{\Gamma}^{\nu}(T\mathcal{U})$. Consider the mapping $\mathscr{X}: \mathcal{W} \rightarrow
\mathrm{Et(\mathscr{G}^{\nu}_{TM})}$, where $\mathcal{W} \subseteq \mathbb{T} \times M$ is open, $\mathscr{G}^{\nu}_{TM}$ denotes the sheave of sections of
$TM$ of class $C^{\nu}$, and $\mathrm{Et(\mathscr{G}^{\nu}_{TM})}$ denotes the etale space of $\mathscr{G}^{\nu}_{TM}$.
Suppose $\mathscr{X}$ satisfies the following properties:

(i)$\mathscr{X}(t,x) \in \mathscr{G}^{\nu}_{TM, x}$;

(ii) for fixed $x$, $t \rightarrow \mathscr{X}(t,x)$ is locally integrable;

(iii) for fixed $t$, $x \rightarrow \mathscr{X}(t,x)$ is continuous in the etale topology.\\
From \cite{Ref8} we know the set of mappings from $\mathcal{W}$ into
$\mathrm{Et(\mathscr{G}^{\nu}_{TM})}$ satisfying (i) (ii) (iii) is an identification with local sections of
$\mathrm{LI\mathscr{G}^{\nu}}(\mathbb{T}, TM)$ over $\mathcal{W}$.  Let $\mathfrak{G}=(M, \mathscr{F})$ be
a $C^{\nu}$-tautological control system. An etale open-loop system is a local section over some $\mathcal{W} \subseteq \mathbb{T} \times M$ of $\mathrm{LI\mathscr{G}^{\nu}(\mathbb{T}, \mathscr{F})}$.
An etale open-loop subsystem is an assignment to each
$\mathcal{W} \subseteq \mathbb{T} \times M$ a subset $\mathcal{G}_{\mathcal{W}} \subseteq
\mathrm{LI\mathscr{G}^{\nu}(\mathbb{T}, \mathscr{F})}(\mathcal{W})$ such that if $\mathcal{W}_1 \subseteq \mathcal{W}_2$,
$r_{\mathcal{W}_{2}, \mathcal{W}_{1}}(\mathcal{G}_{\mathcal{W}_{2}}) \subseteq \mathcal{G}_{\mathcal{W}_{1}}$. we will use $\mathcal{G}_{\mathfrak{G}, \text{pwc}}$ to denote the piecewise constant etale open-loop subsystem for $\mathfrak{G}$.
\begin{definition}
A trajectory for $C^{\nu}$-tautological control system $\mathfrak{G}=(M, \mathscr{F})$ is an absolutely continuous curve $t \rightarrow \xi(t)$ such that
there exists $\mathrm{X}: \mathcal{W} \rightarrow Et(\mathscr{F})$ satisfying (i) (ii) (iii) for some $\mathcal{W} \subseteq \mathbb{T} \times M$ for
which $\xi{'}(t)=ev_{\xi(t)}(\mathrm{X}(t,\xi(t)))$ for almost every $t$, where $ev_x: \mathscr{G}^{\nu}_{TM, x} \rightarrow T_{x}M$ is
$ev_{x}([X]_{x})=X(x)$.
\end{definition}

We say that $x_1 \in M$ is reachable from $x_0 \in M$ at time $T \geq 0$ for $\mathfrak{G}=(M, \mathscr{F})$
if there exists a trajectory $\xi(t)$ for $\mathfrak{G}=(M, \mathscr{F})$ such that
$\xi(0)=x_0$ and $\xi(T)=x_1$. The reachable set from $x_0$ for $\mathfrak{G}=(M, \mathscr{F})$ is $\mathcal{R}_{\mathfrak{G}}(x_0)=\{x \in M \ |\
x \ \text{is reachable from}\ x_0 \ \text{at some time} \ t \geq 0 \  \text{for} \ \mathfrak{G}=(M, \mathscr{F})\}$.
We define $\mathcal{R}_{\mathfrak{G}}(x_0,  T)=\{x \in M \ |\
x \ \text{is reachable from}\ x_0 \ \text{at time} \ T  \  \text{for} \ \mathfrak{G}=(M, \mathscr{F})\}$.
The set of states reachable in time at most $T$ is given by $R_{\mathfrak{G}}(x_0, \leq T)=\cup_{t \in [0,T]}R_{\mathfrak{G}}(x_0, t)$.
We will say that the tautological control system $\mathfrak{G}$ is reachable from $x_0 \in M$ if $\mathcal{R}_{\mathfrak{G}}(x_0)=M.$
We will use $\mathcal{R}_{\mathfrak{G}}(x_0, \mathcal{G}_{\mathfrak{G}, \text{pwc}})$ to denote the set of points that are reachable from
$x_0 \in M$ by  the piecewise constant etale open-loop subsystem for $\mathfrak{G}$.
\begin{definition}
Let $m \in \mathbb{Z}_{\geq 0}$ and $m{'} \in \{0, \mathrm{lip}\}$, let $\nu \in \{m+m{'}, \infty, \omega\}$,
and let $r \in \{\infty, \omega\}$, as required. Let $\mathfrak{G}=(M, \mathscr{F})$ be a $C^{\nu}$-tautological control
system. We say a set $C \subseteq M$ is a reachability set for $\mathfrak{G}$ if for every $(x,z) \in C \times C$ the point
$z$ is reachable from the point $x$ at some time $t \geq 0$ for $\mathfrak{G}=(M, \mathscr{F})$.
\end{definition}
\begin{proposition}\label{prop5}
Let $m \in \mathbb{Z}_{\geq 0}$ and $m{'} \in \{0, \mathrm{lip}\}$, let $\nu \in \{m+m{'}, \infty, \omega\}$,
and let $r \in \{\infty, \omega\}$, as required. Let $\mathfrak{G}=(M, \mathscr{F})$ and $\mathfrak{H}=(N, \mathscr{G})$
be $C^{\nu}$-tautological control systems. Suppose that there exists a trajectory-preserving morphism $(\Phi, \Phi^{\#})$
from $\mathfrak{G}$ to $\mathfrak{H}$ that is global in time, $y_0 \in N$ is such that $y_0 \in \Phi(M)$ and
$\Phi(M) \subseteq \mathcal{R}_{\mathfrak{H}}(y_0, \mathcal{G}_{\mathfrak{H}, \mathrm{pwc}})$, and the fiber $\Phi^{-1}(y_0)$ is a reachability set for
$\mathfrak{G}$. Then for every $x_0 \in \Phi^{-1}(y_0)$ we have
$\mathcal{R}_{\mathfrak{G}}(x_0)=M.$
\end{proposition}
\begin{proof}
The proof is analogous to the proof for Theorem 5 in \cite{Ref7} so we omit it here.\qed
\end{proof}

\begin{theorem}\label{thm2}
Let $m \in \mathbb{Z}_{\geq 0}$ and $m{'} \in \{0, \mathrm{lip}\}$, let $\nu \in \{m+m{'}, \infty\}$,
and let $r=\infty$, as required. Let $\mathfrak{H}=(N, \mathscr{G})$ be a $C^{\nu}$-tautological control system that is
globally generated by a family of linearly independent vector fields. Let $M$ be a $C^{r}$ manifold and $\Phi: M \to N$ be a
proper submersion. Suppose that $y_0 \in \Phi(M)$, $\Phi(M) \subseteq \mathcal{R}_{\mathfrak{H}}(y_0, \mathcal{G}_{\mathfrak{H}, \mathrm{pwc}})$,
and $\Phi^{-1}(y_0)$ is connected. Then there exists a $C^{\nu}$-tautological control system $\mathfrak{G}=(M, \mathscr{F})$
such that there exists a trajectory-preserving morphism $(\Phi, \Phi^{\#})$ from $\mathfrak{G}$ to $\mathfrak{H}$. Besides, $\mathfrak{G}$
is reachable from every $x_0 \in \Phi^{-1}(y_0)$ by $\mathcal{G}_{\mathfrak{G}, \text{pwc}}$.
\end{theorem}
\begin{proof}
By Theorem \ref{thm1} we know there exists a globally generated $C^{\nu}$-tautological control system $\mathfrak{G}=(M, \mathscr{F})$ such that there exists a trajectory-preserving morphism from $\mathfrak{G}$ to $\mathfrak{H}$. Since $\Phi$ is a submersion,
the distribution $\ker T\Phi$ is a smooth distribution with constant rank. Consider the sheaf of sets of $C^{\nu}$-vector fields $\mathscr{H}$ such that for each open $\mathcal{U} \subseteq M$,
\begin{equation}\label{equ1}
\mathscr{H}(\mathcal{U})=\{X \in \Gamma^{\nu}(T\mathcal{U}) \ | \ X(x) \in \ker T\Phi(x), x \in \mathcal{U}\}.
\end{equation}
We can define a presheaf of sets of $C^{\nu}$-vector fields $\overline{\mathscr{F}}$ on $M$ such that
\begin{equation}\label{equ2}
\overline {\mathscr{F}}(\mathcal{U})=\{X \in \Gamma^{\nu}(T\mathcal{U})\ |\ X \in \mathscr{F}(\mathcal{U}) \ \text{or}\ X \in \mathscr{H}(\mathcal{U})\}.
\end{equation}
We claim that the tautological control system $\overline{\mathfrak{G}}(M, \overline{\mathscr{F}})$
satisfies the conditions stated in the theorem. Here the trajectory-preserving morphism can be chosen as the trajectory-preserving morphism from
$\mathfrak{G}$ to $\mathfrak{H}$ above. The remaining thing is to prove that $\overline{\mathfrak{G}}(M, \overline{\mathscr{F}})$
is reachable from every $x_0 \in \Phi^{-1}(y_0)$ by $\mathscr{O}_{\overline{\mathfrak{G}}, \mathrm{pwc}}$.

For $\Phi$ is a proper mapping, by Corollary 1 in \cite{Ref7}, we know the trajectory-preserving morphism from
$\overline{\mathfrak{G}}$ to $\mathfrak{H}$ is global in time. Then according to Proposition 5, it remains to show that the fiber $\Phi^{-1}(y_0)$ is a
reachability set for $\overline {\mathfrak{G}}$. Since $\Phi$ is a submersion, $\Phi^{-1}(y_0)$ is an embedded submanifold of $M$. By assumption it is connected,
so it is also path connected. If we choose arbitrary points $x_1, x_2 \in \Phi^{-1}(y_0)$,
then there exists a piecewise smooth path $\xi: [0,1] \rightarrow \Phi^{-1}(y_0)$ such that $\xi(0)=x_1, \xi(1)=x_2$.
Then $T_{\xi(t)}\Phi(\xi{'}(t))=0$ for each $t \in [0, 1]$. That is, $\xi{'}(t) \in \ker T\Phi(\xi(t))$, for each $t \in [0, 1]$.
According to the definition of $\mathscr{H}$, we know that $\xi(t)$ is a trajectory for $\overline{\mathfrak{G}}$
with respect to a piecewise constant etale open-loop system. That is, $x_2$ is reachable from $x_1$ for $\overline{\mathfrak{G}}$.
Since the points $x_1, x_2 \in \Phi^{-1}(y_0)$ are arbitrary, we have proved that the fiber $\Phi^{-1}(y_0)$ is a reachability set for $\overline {\mathfrak{G}}$
with respect to $\mathcal{G}_{\overline{\mathfrak{G}}, \text{pwc}}$.
Since for each point in $\Phi(M)$, it is also reachable from $y_0$ for $\mathfrak{H}$
with respect to a piecewise constant etale open-loop system, we conclude that $\overline {\mathfrak{G}}$
is reachable from every $x_0 \in \Phi^{-1}(y_0)$ by $\mathcal{G}_{\overline {\mathfrak{G}}, \text{pwc}}$.
\end{proof}
\remark Analogous result in \cite{Ref7} is presented in the context of affine control systems on Euclidean spaces,
where additional assumptions on the drift vector field are given to preserve the structure of affine control systems.
Here under the framework of tautological control systems, we show that the results hold for general manifolds,
not only Euclidean spaces. With respect to discussions on the correspondence between control systems and
tautological control systems in the following subsection, we can translate the above results into the language
of ordinary control systems. Unlike the situations in \cite{Ref7} where a previous result on the regularity properties of controls for global controllable $C^{1}$-control system is quoted in the proof, here we prove directly that $\mathfrak{G}$
is reachable from every $x_0 \in \Phi^{-1}(y_0)$ by piecewise constant etale open-loop subsystem. This is due to the use of sheaf language, where there are no choice of parameterisation by control (see (\ref{equ1})). And similar results can be stated for real analytic tautological control systems by Theorem 2.

The tautological control system $\mathfrak{G}=(M, \mathscr{F})$  given in the above theorem is not globally
generated because the set of $C^{\nu}$-vector fields $\mathscr{H}$ is a sheaf. However, we can  construct
a globally generated tautological control system that  satisfies the requirement of the above theorem as follows.
\begin{definition}
Let $M$ be a smooth manifold. A smooth generalized distribution is a subset $D \subseteq TM$ such that, for each $x_0 \in M$,
the subset $D_{x_0}=D\cap T_{x_0}M$  is a subspace. Besides, there exists a neighbourhood $\mathcal{N}$ of $x_0$ and a family
$(\xi_j)_{j \in J}$ of smooth vector fields, called local generators, on $\mathcal{N}$ such that
\[D_x=\text{span}_{\mathbb{R}}\{\xi_j(x)| j \in J\}\]
for each $x \in \mathcal{N}$. A smooth generalized distribution $D$ is globally finitely generated, if there exists a family
$(\xi_1, \cdots, \xi_k)$ of smooth vector fields on $M$ such that
\[D_x=\text{span}_{\mathbb{R}}\{\xi_1(x), \cdots, \xi_k(x)\}\]
for each $x \in M$.
\end{definition}

\begin{theorem}\cite{Ref12}\label{thm3}
Let $M$ be a connected manifold. Let $D$ be a smooth generalized distribution on $M$.
Then $D$ is globally finitely generated.
\end{theorem}
Since $\ker T\Phi$ is a smooth distribution on $M$, then according to Theorem \ref{thm3}, there exist smooth vector fields
$X_1, \cdots, X_k$ on $M$ such that
\begin{equation}\label{equ5}
 \ker T\Phi(x)=\text{span}_{\mathbb{R}}\{X_{1}(x), \cdots, X_{k}(x)\}.
\end{equation}
Consider the globally generated presheaf $\mathscr{H}$ such that for each open $\mathcal{U} \subseteq M$,
\[\widetilde{\mathscr{H}}(\mathcal{U})=\{\sum_{i=1}^{k} c_{i}X_{i} | {\mathcal{U}},  c_{i} \in \mathbb{R}, i=1, \cdots, k\}.\]
Then we define a presheaf of sets of $C^{\nu}$-vector fields $\widetilde{\mathscr{F}}$ on $M$ such that
\[\widetilde {\mathscr{F}}(\mathcal{U})=\{X \in \Gamma^{\nu}(T\mathcal{U}) \ | \ X \in \mathscr{F}(\mathcal{U})\  \text{or} \ X \in \widetilde{\mathscr{H}}(\mathcal{U})\}.\]
$\widetilde{\mathscr{F}}$ is globally generated. Given arbitrary points $x_1, x_2 \in \Phi^{-1}(y_0)$, then for any piecewise smooth path $\xi: [0,1] \rightarrow \Phi^{-1}(y_0)$ such that $\xi(0)=x_1, \xi(1)=x_2$, locally in a subinterval $\mathbb{T}^{'}$ there exist smooth functions $\alpha_{1}, \cdots, \alpha_{k}$ on open $\mathcal{U} \subseteq M$ such that $\xi^{'}(t)=\sum_{i=1}^{k} \alpha_{i}(\xi(t))X_{i}(\xi(t))$. Let the mapping $\mathscr{X}: \mathcal{W} \rightarrow
\mathrm{Et(\mathscr{G}^{\nu}_{TM})}$ be $\mathscr{X}(t,x)= [\sum_{i=1}^{k} \alpha_{i}(\xi(t))X_{i}]_{x}$ for each
$(t,x) \in \mathbb{T}^{'} \times \mathcal{U}$. It satisfies the conditions (i) (ii) (iii) for the mapping $\mathscr{X}$
above. So $\xi(t)$ is a trajectory for $\widetilde{\mathfrak{G}}=(M, \widetilde{\mathscr{F}})$.
Then for the globally generated tautological control system
$\widetilde{\mathfrak{G}}=(M, \widetilde{\mathscr{F}})$, for every $x_0 \in \Phi^{-1}(y_0)$ we have
$\mathcal{R}_{\widetilde{\mathfrak{G}}}(x_0)=M.$

\remark For a real analytic generalized distribution on $M$, we can not conclude that it is globally finitely generated.
However, with the aid of Cartan's Theorem A (see \cite{Ref13}), we know that for each point $x_{0} \in M$, there exist
a neighbourhood $\mathcal{U}$ of $x_{0}$ and  real analytic global sections $X_{1}, \cdots, X_{k}$ such that
the generalized distribution is spanned by these global sections on $\mathcal{U}$. In the case of real analyticity,
$\ker T\Phi$ is a real analytic distribution on $M$. We can apply the above approach for constructing $\widetilde{\mathfrak{G}}$ to the real analytic systems.
The only difference is that for real analytic systems,
there may not be finitely many global generators for the distribution.

For illustration of Theorem 3 we consider the following examples.

Recall that a fibred manifold is a triple $(E, \pi, M)$ where $E$ and $M$ are manifolds and $\pi: E \rightarrow M$
is a surjective submersion. A local trivialization of $\pi$ around $p \in M$ is a triple $(W_p, F_p, t_p)$ where
$W_p$ is a neighbourhood of $p$, $F_p$ is a manifold and $t_p: \pi^{-1}(W_p) \rightarrow W_p \times F_p$ is a diffeomorphism
satisfying the condition
\[pr_1 \circ t_p=\pi|_{\pi^{-1}(W_p)}.\]
A fibred manifold which has at least one local trivialization around each point of its base space is called locally
trivial and is know as a bundle.
\begin{example}
Let $(E, \pi, M)$ be a bundle with $F_p$ being a compact connected manifold (e.g.principal $SO(n)$-bundle over an orientable
$n$-dimensional Riemannian manifold $M$).
Let $\mathfrak{H}=(M, \mathscr{G})$ be a $C^{\infty}$-tautological control system that is
globally generated by a family of linearly independent vector fields.
Suppose that $\mathfrak{H}$ is reachable from $x_0 \in M$. Then there exists a $C^{\infty}$-tautological control system $\mathfrak{G}=(E, \mathscr{F})$
such that there exists a trajectory-preserving morphism $(\Phi, \Phi^{\#})$ from $\mathfrak{G}$ to $\mathfrak{H}$. Besides, $\mathfrak{G}$
is reachable from every $p_0 \in \pi^{-1}(x_0)$.
\end{example}
\begin{example}
Consider the Mobius band $M$, whose total space may be constructed from the topological space $[0,1] \times (0,1)$
by identifying the points $(0,y)$ and $(1,1-y)$, Giving the quotient space the structure of a 2-dimensional smooth manifold,
the image of the set of points $[0,1] \times \{\frac {1}{2}\}$ under the quotient map is then diffeomorphic
to the circle $S^{1}$, and the projection $[0,1] \times (0,1) \rightarrow [0,1] \times \{\frac{1}{2}\}$
passing to the quotient is a surjective submersion from the Mobius band to the circle.
Let $\mathscr{G}$ be globally generated by $e^{i\theta}u$, $\ u \geq 0$. Then the tautological control system
$\mathfrak{H}=(S^{1}, \mathscr{G})$ is reachable from every point in $S^{1}$. The vector field
$\theta\frac{\partial}{\partial x}$ and the smooth distribution
$\text{span}_{\mathbb{R}}\{\frac{\partial}{\partial y}\}$ on $\mathbb{R}^{2}$ are projectable by the quotient map.
Then we get globally defined complete vector field
$f$ and smooth 1-dimensional distribution $D$ on the Mobius band. Let the presheaf
$\mathscr{F}$ be
defined by $\mathscr{F}(U)=\{X \in \Gamma^{\infty}(TU)\ | \ X=cf|U, \ c \geq 0\ or \ X(x) \in D(x), \forall  x \in U\}$, where
$U$ is an open set of the Mobius band. Then from Theorem 3 we know that there exists a trajectory-preserving morphism from
$\mathfrak{G}=(M, \mathscr{F})$ to $\mathfrak{H}$. And the tautological control system $\mathfrak{G}$
is reachable from every point of the Mobius band.

\end{example}

In the remaining part of this section, we will discuss how the above reachability results can be extended to controllability
results.

We say that a tautological control system $\mathfrak{G}$ is small-time locally controllable from $x_0$ if
there exists $T \ge 0$ such that $x_0 \in \text{int}(R_{\mathfrak{G}}(x_0, \leq t))$ for each $t \in (0,T]$.
\begin{proposition}
Let $m \in \mathbb{Z}_{\geq 0}$ and $m{'} \in \{0, \mathrm{lip}\}$, let $\nu \in \{m+m{'}, \infty\}$,
and let $r=\infty$, as required. Let $\mathfrak{H}=(N, \mathscr{G})$ be a $C^{\nu}$-tautological control system that is
globally generated by a family of linearly independent vector fields. Let $M$ be a $C^{r}$ manifold and $\Phi: M \to N$ be a
proper submersion. Suppose that $y_0 \in \Phi(M)$, and $\mathfrak{H}$ is small-time locally controllable from $y_0$,
and $\Phi^{-1}(y_0)$ is connected. Then there exists a $C^{\nu}$-tautological control system $\mathfrak{G}=(M, \mathscr{F})$
such that there exists a trajectory-preserving morphism $(\Phi, \Phi^{\#})$ from $\mathfrak{G}$ to $\mathfrak{H}$. Besides, $\mathfrak{G}$
is small-time locally controllable  from every $x_0 \in \Phi^{-1}(y_0)$.
\end{proposition}
\begin{proof}
Consider the tautological control system $\overline{\mathfrak{G}}(M, \overline{\mathscr{F}})$ as given in equation (\ref{equ2}).
We claim that $\overline{\mathfrak{G}}$ is small-time locally controllable from  every $x_0 \in \Phi^{-1}(y_0)$.

Let $x_0 \in \Phi^{-1}(y_0)$. Since $\mathfrak{H}$ is small-time locally controllable from $y_0$,
we know that there exists $T > 0$ such that $y_0 \in \text{int}(R_{\overline{\mathfrak{G}}}(y_0, \leq t))$ for each $t \in (0,T]$.
Then given $t \in (0,T]$, there exists an open set $U$ such that $y_0 \in U \subseteq \text{int}(R_{\overline{\mathfrak{G}}}(y_0, \leq t))$.
Since $\Phi$ is a submersion, there exists a local neighbourhood $V$ of $x_0$ such that $\Phi(V) \subseteq U$. For any $x \in V$,
we have $\Phi(x) \in U$. Then $\Phi(x)$ is reachable from $y_0$ in time at most $t$. By the mapping $\Phi^{\#}$, we know
that $x$ is reachable from some point $\bar x \in \Phi^{-1}(y_0)$ in time at most $t$. On the other hand, since $\Phi^{-1}(y_0)$ is connected,
then for arbitrary $\epsilon > 0$, there exists a piecewise smooth path $\xi: [0,\epsilon] \rightarrow \Phi^{-1}(y_0)$ such that $\xi(0)=x_0, \xi(\epsilon)=\bar x$,
where $\xi(t)$ is a trajectory for $\overline{\mathfrak{G}}$ with respect to a piecewise constant etale open-loop system. So $x$ is reachable from $x_0$ in time at most $t+\epsilon$. That is,
$x_0 \in \text{int}(R_{\overline{\mathfrak{G}}}(x_0, \leq t+\epsilon))$. So $\overline{\mathfrak{G}}$ is small-time
locally controllable from  $x_0$  due to the arbitrarity of $t$ and $\epsilon$.
\end{proof}
\subsection{Correspondence between control systems and tautological control systems}
\label{sec:5}
In this section, we will investigate the relations between the notion of lifting and morphisms of tautological control systems. First we recall the notion of lifting control systems.

Let $m \in \mathbb{Z}_{\geq 0}$ and $m{'} \in \{0, \mathrm{lip}\}$, let $\nu \in \{m+m{'}, \infty\}$, and let $r=\infty$, as required.
Recall the used notations in Definition 6, then we have
\begin{definition}\cite{Ref7}
 Let $\Sigma_1=(M_1, F_1, \mathfrak{C_1})$ and $\Sigma_2=(M_2, F_2, \mathfrak{C_2})$ be two $C^{\nu}$-control systems and let $\Phi: M_1 \rightarrow M_2$  be a $C^{r}$ mapping. We say that $\Sigma_2$ is liftable to $\Sigma_1$ if there is a mapping $l: \mathfrak{C_2} \rightarrow \mathfrak{C_1}$ such that for every $u \in  \mathfrak{C_2}$ the $C^{\nu}$-vector fields $F_1^{l(u)}$ and $F_2^{u}$ are $\Phi$-related.
\end{definition}
Given two $C^{\nu}$-control systems \[\Sigma_1=(M_1, F_1, \mathfrak{C_1}), \Sigma_2=(M_2, F_2, \mathfrak{C_2}),\]
we associate the $C^{\nu}$-tautological control systems
\[\mathfrak{G}_{\Sigma_1}=(M_1, \mathscr{F}_{\Sigma_1}), \mathfrak{G}_{\Sigma_2}=(M_2, \mathscr{F}_{\Sigma_2}),\]
as introduced in Section \ref{sec:1}.
\begin{proposition}\label{prop6}
Suppose the $C^{\nu}$-vector fields  $F_{2}^{u}, u \in \mathfrak{C_2}$ are linearly independent, then the $C^{\nu}$-control system $\Sigma_2$ is liftable to $\Sigma_1$  if and only if there exists a trajectory-preserving morphism from the $C^{\nu}$-tautological control system $\mathfrak{G}_{\Sigma_1}$ to $\mathfrak{G}_{\Sigma_2}$.
\end{proposition}
\begin{proof}
To prove the "only if" part, suppose that $\Sigma_2$ is liftable to $\Sigma_1$ via $\Phi: M_1 \rightarrow M_2$, then there exists a mapping $l: \mathfrak{C_2} \rightarrow \mathfrak{C_1}$ such that
\begin{equation}\label{equ3}
T_{x}\Phi(F_1(x, l(u)))=F_2(\Phi(x), u),
\end{equation}
for each $x \in M_1, u\in \mathfrak{C_2}$.  For each open $\mathcal{V} \subseteq M_2$, let $\Phi^{\#}_{\mathcal{V}}:
\mathscr{F}_{\Sigma_2}(\mathcal{V}) \rightarrow \Phi_{*}\mathscr{F}_{\Sigma_1}(\mathcal{V})$ be defined as
\[\Phi^{\#}_{\mathcal{V}}(F_2^{u}|\mathcal{V})=F_1^{l(u)}|\Phi^{-1}(\mathcal{V}).\]
Since the $C^{\nu}$-vector fields  $F_{2}^{u}, u \in \mathfrak{C_2}$ are linearly independent, then the mapping $u \rightarrow F_2^{u}|\mathcal{V}$ from $\mathfrak{C_2}$ to $\Gamma^{\nu}(T\mathcal{V})$ is injective. So $\Phi^{\#}_{\mathcal{V}}$ is well defined.  $(\Phi, \Phi^{\#})$ satisfies condition (ii) in Definition 2 becuase $F_{2}^{u}, u \in \mathfrak{C_2}$ are linearly independent. This means that $(\Phi, \Phi^{\#})$ is a morphism from $\mathfrak{G}_{\Sigma_1}$ to $\mathfrak{G}_{\Sigma_2}$. It is also a trajectory-preserving morphism because of (\ref{equ3}).
To prove the "if" part, let $(\Phi, \Phi^{\#})$ be trajectory-preserving morphism from $\mathfrak{G}_{\Sigma_1}$ to $\mathfrak{G}_{\Sigma_2}$.
Then for every $u \in \mathfrak{C_2}$, there exists  ${\bar u} \in \mathfrak{C_1}$, such that
\begin{equation}
T_x\Phi(F_1^{\bar u})=F_2^{u}(\Phi(x)),
\end{equation}
for each $x \in M_1$.
Let $l: \mathfrak{C_2} \rightarrow \mathfrak{C_1}$  be $l(u)=\bar u$. Then according to equation (5), we know that
$\Sigma_2$ is liftable to $\Sigma_1$ with lifting function $l: \mathfrak{C_2} \rightarrow \mathfrak{C_1}$. This completes
the proof.
\end{proof}
\remark Let $\Sigma_1$ and $\Sigma_2$ be two control-affine systems with $\Sigma_i=(M_i, F_i, \mathbb{R}^{m_i})$.
Suppose that  $F_2(y,u)=f_0(y)+\sum_{a=1}^{m_2}u^{a}f_a(y)$, where $f_0, f_1, \cdots, f_k$ are linear independent.
Then the $C^{\nu}$-control system $\Sigma_2$ is liftable to $\Sigma_1$ via an affine mapping $l: \mathbb{R}^{m_2} \rightarrow \mathbb{R}^{m_1}$,
if and only if there exists a trajectory-preserving morphism from the $C^{\nu}$-tautological control system $\mathfrak{G}_{\Sigma_1}$ to $\mathfrak{G}_{\Sigma_2}$.

Given two globally generated $C^{\nu}$-tautological control systems \[\mathfrak{G}_1=(M_1, \mathscr{F}_1), \mathfrak{G}_2=(M_2, \mathscr{F}_2),\] we can associate the $C^{\nu}$-control systems
\[\Sigma_{\mathfrak{G}_1}=(M_1, F_{\mathscr{F}_1}, \mathfrak{C}_{\mathscr{F}_1}), \Sigma_{\mathfrak{G}_2}=(M_2, F_{\mathscr{F}_2}, \mathfrak{C}_{\mathscr{F}_2})\] as introduced in Section \ref{sec:1}.
\begin{proposition}
Suppose that $\mathfrak{G}_2$  is globally generated by a family of linearly independent vector fields. There exists a trajectory-preserving morphism from $\mathfrak{G}_1$ to $\mathfrak{G}_2$  if and only if the control system $\Sigma_{\mathfrak{G}_2}$ is liftable to $\Sigma_{\mathfrak{G}_1}$.
\end{proposition}
\begin{proof}
The proposition can be proved directly. It can also be seen as a corollary of Proposition \ref{prop6}. This can be explained as follows.  By assumption, for each
$X \in \mathfrak{C}_{\mathscr{F}_2}= \mathscr{F}_2(M_2)$, $F_{\mathscr{F}_2}^{X}=X$ are linearly independent. Then according to Proposition \ref{prop6}, the control system $\Sigma_{\mathfrak{G}_2}$ is liftable
to $\Sigma_{\mathfrak{G}_1}$ if and only if there exists a trajectory-preserving morphism from
the $C^{\nu}$-tautological control system $\mathfrak{G}_{\Sigma_{\mathfrak{G}_1}}$ to $\mathfrak{G}_{\Sigma_{\mathfrak{G}_2}}$. On the other hand,
from Proposition 3 we know that $\mathfrak{G}_{\Sigma_{\mathfrak{G}_i}}=\mathfrak{G}_i, i=1,2$. This yields the result.
\end{proof}
\section{Control systems of second-order type}
\label{sec:6}
This section contains an example of second-order type control systems as an application of our results
in Section \ref{sec:2}.  Compared with  Theorem 11 in \cite{Ref7}, we get supplementary results due to the special structure properties of this class of control systems.
Let's first introduce the definition of control systems of second-order type.
\begin{definition}
A second-order type control system is a $C^{\infty}$ control-affine system
$\Upsilon^{'}(t)=f_0(\Upsilon(t))+\sum_{i=1}^{k}u^{i}f_{i}(\Upsilon(t)), u^i \in \mathbb{R}, i=1, \cdots, k$
on the tangent bundle $TQ$ of a  connected smooth manifold $Q$ such that $T_{y}\pi_{TQ}(f_0)=y$ and $T_{y}\pi_{TQ}(f_i)=0, i=1, \cdots, k$,
for all $y \in TQ$. Here $f_0, f_1, \cdots, f_m$ are smooth vector fields on $TQ$. A second-order type tautological control system is a $C^{\infty}$-tautological control system  $\mathfrak{G}$
such that $\mathfrak{G}=\mathfrak{G}_{\Sigma}$  where $\Sigma$ is a second-order type control system.
\end{definition}
In terms of coordinates $(x^{i}, y^{i})$ where $y^{i}$ are the canonical fiber coordinates corresponding to coordinates
$x^{i}$ on $Q$, a second-order type control system has the following local form
\begin{eqnarray}\label{equ4}
&&{{x^i}}{'}=y^{i},\nonumber\\
&&{{y^i}}{'}=\Gamma^{i}(x, y)+\sum_{j=1}^{k}u^{j}g_{j}^{i}(x,y).
\end{eqnarray}
The class of second-order type control systems is an attractive subject of research,
not only because they are abundant in real lift but also they offer very interesting mathematical problems. Questions like controllability, accessibility, stabilization provide the starting point for very interesting mathematical theories. A typical class of second-order type control system is  Lagrangian mechanical control system (see \cite{Ref14}), for which the control properties can be characterized by using the underlying geometric structures.

In this section, we will restrict the  main questions introduced in Section \ref{intro} to the class of
second-order type control systems. That is,  given a  second-order type tautological control system
$\mathfrak{H}$ on the tangent bundle $TQ$ of the configuration manifold $Q$, and a mapping $\phi: P \rightarrow Q$
where $P$ is a smooth manifold,
we want to look for a second-order type tautological control system $\mathfrak{G}$ on the tangent bundle $TP$ of the configuration manifold
$P$ such that there is a trajectory-preserving morphism from $\mathfrak{G}$ to $\mathfrak{H}$.  Besides, we will provide sufficient conditions such that the reachability of $\mathfrak{H}$ can imply the reachability of $\mathfrak{G}$.  Here we state the questions under
the framework of tautological control systems to deal with distinctions between local and global, and to lift the control systems such that it is well-defined both locally and globally.

Given a second-order type control system
\begin{equation}\label{equ8}
\Sigma_2: \Upsilon{'}(t)=g_0(\Upsilon(t))+\sum_{i=1}^{k}u^{i}g_{i}(\Upsilon(t)), u^i \in \mathbb{R}, i=1, \cdots, k
\end{equation}
on the tangent bundle $TQ$ of  the configuration manifold $Q$, in the following we will suppose that $g_1, \cdots, g_k$ are linearly independent.
\begin{proposition}\label{prop8}
Let $\Sigma_2$ be a second-order type control system on the tangent bundle $TQ$ of the configuration manifold $Q$.
Let $\mathfrak{H}_{\Sigma_2}=(TQ, \mathscr{F}_{\Sigma_2})$ be its associated second-order type tautological control
system. Let $\phi: P \rightarrow Q$ be a submersion, where $P$ is a smooth manifold. Then there exists a second-order
type tautological control system $\mathfrak{G}_{\Sigma_1}=(TP, \mathscr{F}_{\Sigma_1})$
such that there is a trajectory-preserving morphism from $\mathfrak{G}_{\Sigma_1}$ to $\mathfrak{H}_{\Sigma_2}$.
\end{proposition}
\begin{proof}
Since $\Phi$ is a submersion, for each point $x_{\alpha}\in P$, there exists a neighbourhood $\mathcal{U}_{\alpha}$ of
$x_{\alpha}$, and a smooth vector field $f_0^{\alpha}$ on $T\mathcal{U}_{\alpha}$ such that
$T_{y}T\phi(f_0^{\alpha})=g_0, T_{y}\pi_{TP}(f_0^{\alpha})=y$, for each $y \in T\mathcal{U}_{\alpha}$.
By using  partition of unity subordinated to the open cover $\{\mathcal{U}_{\alpha}\}$ and
by the linearity of the derivative of a differentiable mapping, we can get a smooth vector field $f_0$ on $TP$ such that $T_{y}T\phi(f_0)=g_0, T_{y}\pi_{TP}(f_0)=y$, for each $y \in TP$.  Similarly, we can get smooth vector fields
$f_1, \cdots, f_k$ on $TP$  such that $T_{y}T\phi(f_i)=g_i, T_{y}\pi_{TP}(f_i)=0, i=1, \cdots, k$, for each $y \in TP$.
Let $T\phi^{\#}(g_0+\sum_{i=1}^{k}u^{i}g_{i})=f_0+\sum_{i=1}^{k}u^{i}f_{i}$, for each $u_i \in \mathbb{R}, i=1, \cdots, k$.
Consider the second-order type control system
$\Sigma_1: \eta{'}(t)=f_0(\eta(t))+\sum_{i=1}^{k}v^{i}f_{i}(\eta(t)), v^i \in \mathbb{R}, i=1, \cdots, k$
on the tangent bundle $TP$ of the configuration manifold $P$. Since
$g_1, \cdots, g_k$ are linearly independent, we know that  $(T\phi, T\phi^{\#})$ defines a trajectory-preserving
morphism from $\mathfrak{G}_{\Sigma_1}=(TP, \mathscr{F}_{\Sigma_1})$ to $\mathfrak{H}_{\Sigma_2}$.
\end{proof}
\begin{theorem}\label{thm6}
Let $\Sigma_2$ be a second-order type control system on the tangent bundle $TQ$ of  the configuration manifold $Q$.
Let $\mathfrak{H}_{\Sigma_2}=(TQ, \mathscr{F}_{\Sigma_2})$ be its associated second-order type tautological control
system. Let $\phi: P \rightarrow Q$ be a submersion, where $P$ is a smooth manifold. Suppose that $z_0 \in \phi(P)$
is such that $g_0(z_0, 0)=0$ and $\phi^{-1}(z_0)$ is connected, $T(\phi(P)) \subseteq \mathcal{R}_{\mathfrak{H}_{\Sigma_2}}((z_0,0), \mathcal{G}_{\mathfrak{H}_{\Sigma_2}, \mathrm{pwc}})$.
Then there exists a  second-order type tautological control system $\mathfrak{G}_{\Sigma_1}=(TP, \mathscr{F}_{\Sigma_1})$ such that there exists a trajectory-preserving morphism from $\mathfrak{G}_{\Sigma_1}$ to $\mathfrak{H}_{\Sigma_2}$,
and if the morphism is global in time, then $\mathfrak{G}_{\Sigma_1}$ is reachable from every $(x_0, v_0) \in T(\phi^{-1}(z_0))$.
\end{theorem}
\begin{proof}
Since $\phi$ is a submersion, then according to (\ref{equ5}), there exist smooth vector fields $X_1, \cdots, X_l$
on $P$ such that \[ \ker T\phi(x)=\text{span}_{\mathbb{R}}\{X_{1}(x), \cdots, X_{l}(x)\},\]
for all $x \in P$. Consider now the second-order type control system $\Sigma_1$
\[ \eta{'}(t)=f_0(\eta(t))+\sum_{i=1}^{k}v^{i}f_{i}(\eta(t))+\sum_{j=1}^{l}\omega^{j}X_j^{\text{vlft}}(\eta(t)),\]
$v^i, \omega^j \in \mathbb{R}$ for $i=1, \cdots, k, j=q, \cdots, l$,
where $f_0, f_1, \cdots, f_k$ are the vector fields on $TP$ given as in Proposition \ref{prop8} and $X_j^{\text{vlft}}$ denotes the vertical lift of  the vector field $X_j$ on $P$ (see \cite{Ref14}).

We claim that $T\phi^{-1}(z_0, 0)$ is a reachability set for $\mathfrak{G}_{\Sigma_1}$. To prove it,
for $\phi^{-1}(z_0)$ is an embedded submanifold of $P$, denote the embedding map by $i: \phi^{-1}(z_0) \rightarrow P$.
We have $Ti(T(\phi^{-1}(z_0)))=(T\phi)^{-1}(z_0, 0)$.  Choose arbitrary points $(x_1, v_1), (x_2, v_2) \in (T\phi)^{-1}(z_0, 0)$. Since $\phi^{-1}(y_0)$ is connected,  there exists a piecewise smooth path $\xi: [0,1] \rightarrow \phi^{-1}(y_0)$ such that $i(\xi(0))=x_1, Ti(\xi{'}(0))=v_1, i(\xi(1))=x_2, Ti(\xi{'}(1))=v_2$.
Let $\Xi(t)= (i(\xi(t)), Ti(\xi{'}(t))) \in T\phi^{-1}(z_0, 0)$. Since $T_y\pi_{TP}(f_0)=y$, we have
\begin{equation}\label{equ6}
T\pi_{TP}(\Xi{'}(t)-f_0(\Xi(t))= Ti(\xi{'}(t))-Ti(\xi{'}(t))=0.
\end{equation}
Besides,
\begin{equation}\label{equ7}
TT\phi(\Xi{'}(t)-f_0(\Xi(t))=0-g_0(z_0, 0)=0-0=0.
\end{equation}
According to (\ref{equ6}) and (\ref{equ7}), we have
\[\Xi{'}(t)-f_0(\Xi(t)) \in \text{span}_{\mathbb{R}}\{X_1^{\text{vlft}}(\Xi(t)), \cdots, X_l^{\text{vlft}}(\Xi(t)) \},\]
for $t \in [0,1]$. It follows that there exist piecewise smooth functions $\omega^j: [0,1] \rightarrow \mathbb{R}, j=1, \cdots, l$, such that
\[\Xi{'}(t)=f_0(\Xi(t))+\sum_{j=1}^{l}\omega^{j}(t)X_j^{\text{vlft}}(\Xi(t)).\]
Since the piecewise smooth functions $\omega^j$ satisfy  $\omega^j \in \mathrm{L}^{\infty}_\mathrm{loc}([0,1]; \mathbb{R})$, then according to Proposition \ref{prop2},
$\Xi(t)$ is a trajectory for $\mathfrak{G}_{\Sigma_1}$. This means $T\phi^{-1}(z_0, 0)$ is a reachability set for $\mathfrak{G}_{\Sigma_1}$.
Then we know that $\mathfrak{G}_{\Sigma_1}$ is reachable from every $(x_0, v_0) \in T(\phi^{-1}(z_0))$ by Proposition 5.
\end{proof}
\remark
Consider the second-order type control system (\ref{equ8}).
According to Theorem 11 in \cite{Ref7}, locally there exists an affine control system $\Sigma$ on $TP$ with $2n-2m+k$ inputs
such that it satisfies the requirement of Theorem 11. Here $n-m$ is the rank of the distribution $\ker T\phi$. This result is improved by our Theorem \ref{thm6}  in the following aspects.
First,  $\Sigma$  can be required to admit the structure of second-order type. Second, locally the number of inputs can be reduced to $n-m+k$. Third, a globally defined control system is derived.
\example
Given an  affine connection control system $\Sigma_2$ on the configuration manifold $Q$  (see \cite{Ref14}),
 \[\nabla_{\gamma{'}(t)}\gamma{'}(t)=\sum_{r=1}^{q}u^rg_{r}(\gamma(t)),\]
it can be written as an affine control system on  $TQ$
\[\Upsilon{'}(t)=S(\Upsilon(t))+\sum_{r=1}^{q}u^rg_{r}^{\text{vlft}}(\gamma(t)).\]
Using the same reasoning as above, we obtain an affine connection control system $\Sigma_1$
\[\nabla_{\eta{'}(t)}\eta{'}(t)=\sum_{j=1}^{n-m+q}v^jf_{j}(\eta(t)),\]
on the configuration manifold $P$ such that $\Sigma_2$ can be liftable to $\Sigma_1$ via a surjective submersion
$\phi: P \rightarrow Q$. Furthermore, suppose that $\phi^{-1}(z_0)$  is connected for $z_0 \in Q$
and  $\Sigma_2$ can be liftable to $\Sigma_1$ globally in time,
then if $\Sigma_2$  is controllable from  the zero velocity point of $z_0$, we can assert that $\Sigma_1$  is controllable from every  $(x, v) \in T(\phi^{-1}(z_0))$.



\end{document}